\documentclass[11pt]{article}

\usepackage{amssymb, amsmath, amsthm, latexsym}
\usepackage{fullpage, color}
\usepackage{algpseudocode, algorithm}

\newcommand{\R}{\mathbb{R}}
\newcommand{\eqdef}{\stackrel{\text{def}}{=}}

\newcommand{\ve}[2]{\left\langle #1 , #2 \right\rangle}
\def\<#1,#2>{\left\langle #1,#2\right\rangle}

\newtheorem{theorem}{Theorem}

\theoremstyle{plain}
\newtheorem{prop}[theorem]{Proposition}
\newtheorem{lem}[theorem]{Lemma}

\theoremstyle{plain}
\newtheorem{ass}[theorem]{Assumption}

\begin{document}

\title{Simple Complexity Analysis of Simplified Direct Search}
\author{Jakub Kone\v{c}n\'{y} \footnote{School of Mathematics,  University of Edinburgh, United Kingdom (e-mail: j.konecny@sms.ed.ac.uk)} \qquad \qquad  Peter Richt\'{a}rik \footnote{School of Mathematics,  University of Edinburgh, United Kingdom (e-mail: peter.richtarik@ed.ac.uk) \qquad \qquad
The work of both authors was supported by the Centre for Numerical Algorithms and Intelligent Software (funded by EPSRC grant EP/G036136/1 and the Scottish Funding Council). The work of J.K. was also supported by a Google European Doctoral Fellowship and the work of P.R. was supported by the EPSRC grant  EP/K02325X/1.}
\\\\
 {\em School of Mathematics}\\
{\em  University of Edinburgh}
\\
{\em United Kingdom}}

\date{September 30, 2014 (Revised\footnote{In this revision (minor revision in SIAM J Opt) the main results stay the same, but we have further simplified some proofs, added some new results and further streamlined the exposition.} : November 13, 2014)}

\maketitle

\begin{abstract} 
We consider the problem of unconstrained minimization of a smooth function in the derivative-free setting using. In particular, we propose and study a simplified variant of the direct search method (of direction type), which we call {\em simplified direct search (SDS)}. Unlike standard direct search methods, which depend on a large number of parameters that need to be tuned, SDS depends on a single scalar parameter only. 




Despite relevant research activity in direct search methods spanning several decades,  complexity guarantees---bounds on the number of function evaluations needed to find an approximate solution---were not established until very recently.  In this paper we give a surprisingly  {\em brief} and {\em unified analysis} of SDS for nonconvex, convex and strongly convex  functions. We match the existing complexity results for direct search in their dependence on the problem dimension ($n$) and error tolerance ($\epsilon$), but the overall  bounds are  simpler,  easier to interpret, and have better dependence on other problem parameters. In particular, we show that  for the set of directions formed  by the standard coordinate vectors and their negatives, the number of function evaluations needed to find an $\epsilon$-solution is $O(n^2 /\epsilon)$ (resp.  $O(n^2 \log(1/\epsilon))$) for the problem of minimizing a convex (resp. strongly convex) smooth function. In the nonconvex smooth case, the bound is $O(n^2/\epsilon^2)$, with the goal being the reduction of the norm of the gradient below $\epsilon$.

\bigskip

\end{abstract}


\newpage
\section{Introduction}

In this work we study the problem of unconstrained minimization of a smooth function $f: \R^n \to \R$:
\begin{equation} \label{eq:main}\min_{x \in \R^n} f(x).\end{equation}
We assume that we only have access to a function evaluation oracle; that is, we work in the derivative-free setting.  In particular, we study a simplified variant of the direct search method of directional type, which we call {\em Simplified Direct Search} (SDS), and establish complexity bounds for nonconvex, convex and strongly convex objective functions $f$. That is, we prove bounds on the number of function evaluations  which lead to the identification of an approximate solution of the optimization problem.


 Despite the effort by a community of researchers spanning more than a half century \cite{hookejeeves, wenpositivebasis, Torczon97, dolanlewis, audet02-gps, kolda03, abramson06, audet06-meshadapt, audet-orban06, DS-discont}, complexity bounds for direct search have not been established until very recently in a sequence of papers by Vicente and coauthors \cite{vicentenonsmooth, vicente, Garmanjani-Vicente12, zaikun14}.  To the best of our knowledge, the first complexity result was established  by Vicente   in the case when $f$ is smooth (and possibly nonconvex)  \cite{vicentenonsmooth}, with the goal being the identification of a point $x$ for which $\|\nabla f(x)\|\leq \epsilon$. In this  work, it was  shown that direct search will find such a point in $O(n^2/\epsilon^2)$ function evaluations. Subsequently,  Dodangeh and Vicente \cite{vicente} studied the case when $f$ is convex (resp. strongly convex) and proved the complexity bound $O(n^2/\epsilon)$ (resp.  $O(n^2\log (1/\epsilon))$). Garmanjani and Vicente \cite{Garmanjani-Vicente12} established an  $O(n^2/\epsilon^4)$ bound in the case when $f$ is nonsmooth and nonconvex. Finally,  Gratton, Royer, Vicente and Zhang \cite{zaikun14} studied direct search with probabilistic descent.

\subsection{Simplification of Direct Search}  

The direct search method, in its standard form studied in the literature (e.g., \cite{vicentenonsmooth, vicente}), depends on a large number of parameters and settings. A typical setup:
\begin{itemize}
\item uses 6 scalar parameters: $c>0$ (forcing constant), $p>1$ (exponent of the forcing function), $\alpha_0>0$ (initial stepsize), $0<\beta_1\leq  \beta_2<1$ (stepsize decrease factors), $\gamma\geq 1$ (stepsize increase factor),
\item includes a ``search step'' and  a ``poll step'',
\item allows for the set of search directions (in the poll step) to change throughout the iterations (as long as their cosine measure is above a certain positive number), 
\item and allows the search directions to have arbitrary positive lengths.
\end{itemize} 

Admittedly, this setup gives the method  flexibility, as in certain situations one may wish to utilize prior knowledge or experience about the problem at hand to find a mix of parameters that works better than another mix. On the other hand, this flexibility is problematic as one has to decide on how to choose these parameters, which might be a daunting task even if one has prior experience with various parameter choices. The issue of optimal choice of parameters is  not discussed in existing literature. 

Let us now look at the situation from the point of view of complexity analysis. While still present in the standard form of direct search, it  has been recognized  that the ``search step'' does not influence the complexity analysis  \cite{vicentenonsmooth, vicente}. Indeed, this step merely provides a plug for the inclusion of  a clever heuristic, if one is available.  Moreover, direct search methods use a general forcing function\footnote{Direct search accepts a new iterate only if the function value has been improved by at least $
\rho(\alpha)$, where $\alpha>0$ is the current stepsize.}, usually of the form $\rho(\alpha)=c\alpha^p$, where $c>0$ is a forcing constant and $p >1$.  It can be easily inferred from the  complexity results of Vicente covering the nonconvex case \cite{vicentenonsmooth}, and it is made explicit in \cite{vicente},  that  the  choice $p=2$  gives a better bound than other choices. Still, it is customary to consider the more general setup with arbitrary $p$. From the complexity point of view, however, one  does not need to consider the search step, nor is there any need for flexibility in the choice of $p$. Likewise, complexity does not improve by the inclusion of the possibility of replacing the set of search directions at every iteration  as it merely depends on the smallest of the cosine measures of all these sets. In this sense, one can simply fix a single set of directions before the method is run and use that throughout.  

The question that was the starting point of this paper was:

\begin{quote}Which of the parameters and settings of direct search (in its standard form  studied in the literature) are {\em required} from the complexity point of view? If we remove the unnecessary parameters, will it be possible to gain more insight into the workings of the method and possibly provide compact proofs leading to simpler and better bounds?  \end{quote}

In contrast with the standard approach, SDS depends on a single parameter only (in relation to standard direct search, we fix $\beta_1=\beta_2=0.5$, $p=2$, $\gamma=1$). As presented in Section~\ref{sec:DSO}, our method seems to depend on two parameters, initial stepsize $\alpha_0>0$ and  forcing constant\footnote{We shall see that it is optimal to choose $c$ to be equal to the Lipschitz constant of the gradient of $f$. If we have some knowledge of this, which one can not assume in the derivative-free setting, it makes it easier to set $c$.} $c>0$.  However, we show in Section~\ref{sec:init}  that one can, at very low cost, identify suitable $\alpha_0$ {\em or} $c$ automatically, removing the dependence on one of these parameters. Moreover, we exclude the (irrelevant-to-complexity-considerations) ``search step'' altogether, keep just a single set of search directions $D$ throughout the iterative process, and keep them all of unit length (none of the extra flexibility leads to better complexity bounds). In fact, we could have gone even one step further and fixed $D$ to be a very particular set: $D_+=\{\pm e_i, \; i=1,2,\dots,n\}$, where $e_i$ is the $i$th unit coordinate vector (or a set obtained from this by a rotation). Indeed, in all our complexity results, the choice of $D$ enters the complexity via the fraction $|D|/\mu^2$, where $\mu$ is the cosine measure of $D$ (we shall define this in Section~\ref{sec:directions}). We conjecture that  $|D|/\mu^2 \geq n^2$ for any set $D$ for which $\mu>0$. It can be easily seen that this lower bound is achieved, up to a factor of 2, by $D_+$: $|D_+|/\mu^2 = 2n^2$. However, we still decided to keep the discussion in this paper  general in terms of the set of directions $D$ and formulate the algorithm and results that way --  we believe that allowing for arbitrary directions is necessary to retain the spirit of direct search.   As most optimization algorithms, including direct search, SDS also depends on the choice of an initial point $x_0$.






\subsection{Outline}

The paper is organized as follows. We first summarize the contributions of this paper (Section~\ref{sec:contributions}) and then describe our algorithm (Section~\ref{sec:DSO}). In Section~\ref{sec:init} we propose three initialization strategies for the method. Subsequently, in Section~\ref{sec:directions} we formulate our assumptions on the set of directions (positive spanning set). In Section~\ref{sec:main_results} we state and prove our main results, three complexity theorems covering the nonconvex, convex and strongly convex cases. Finally, we conclude in Section~\ref{sec:conclusion}.

\section{Contributions} \label{sec:contributions}


In this section we highlight some of the contributions of this work.

\textbf{1. Simplified algorithm.} We study a novel variant of direct search, which we call {\em simplified direct search} (SDS). While SDS retains the spirit of ``standard'' direct search in the sense that it works with an arbitrary set of search directions forming a positive spanning set, it {\em depends on a single parameter} only: either the forcing constant $c$, or the initial stepsize parameter $\alpha_0$ -- depending on which of two novel initialization strategies is employed.

\textbf{2. Two new initialization strategies.} We describe three initialization strategies, two of which are novel and serve the purpose of automatic tuning of certain  parameters of the method (either the forcing constant or the initial stepsize). The third initialization strategy is {\em implicitly} used by  standard direct search (in fact, it is equivalent to running direct search until the first unsuccessful iteration). However, we argue that this strategy is not particularly efficient, and that it does not remove any of the parameters of the method.

\textbf{3. Simple bounds.} As a byproduct of the simplicity of our analysis we obtain compact and easy to interpret complexity bounds, with small constants.  In Table~\ref{tab:summary}  we summarize selected complexity results (bounds on the number of function evaluations) obtained in this paper. In addition to what is contained therein, we also give bounds on $\|\nabla f(x_k)\|$ in the convex and strongly convex cases, and a bound on $\|x_k-x_*\|$ in the strongly convex case -- see Theorems \ref{thm:convex} and \ref{thm:strconvex}.

{

\begin{table}[!h]
\begin{center}

\footnotesize

\begin{tabular}{|c|c|c|c|c|}
\hline
Assumptions on $f$ &  Goal & 
\begin{tabular}{c}
Complexity \\
($c=O(1))$
\end{tabular} 
&  \begin{tabular}{c}
Complexity\\
 ($c=L/2$) 
 \end{tabular}
 & Thm \\
\hline
& & & & \\
\begin{tabular}{c}
no additional \\
assumptions  
\end{tabular}
& $\|\nabla f(x)\|<\epsilon$ & 
$O\left(\frac{n^2 L^2 (f(x_0)-f^*)}{\epsilon^2}\right)$
&
$O\left(\frac{n^2 L (f(x_0)-f^*)}{\epsilon^2}\right)$ & \ref{thm:nonconvex} \\
& & & & \\
\hline
& & & & \\
\begin{tabular}{c}
convex\\
$\exists$ minimizer $x_*$\\
$R_0<+\infty$
\end{tabular}  
& $f(x)-f(x_*)\leq \epsilon$ &  
$O\left(\frac{n^2 L^2 R_0^2}{\epsilon}\right)$
&
$O\left(\frac{n^2 L R_0^2}{\epsilon}\right)$ & \ref{thm:convex}\\
&& &  &\\
\hline
&& & &\\
\begin{tabular}{c}$\lambda$-strongly\\
 convex 
\end{tabular} 
 & $f(x)-f(x_*) \leq \epsilon$ & 
$O\left(\frac{n^2 L^2}{\lambda} \log\left(\frac{nL^2\alpha_0^2}{\lambda \epsilon}\right)\right)$
&
$O\left(\frac{n^2 L}{\lambda}\log\left(\frac{nL^2\alpha_0^2}{\lambda \epsilon}\right)\right)$ & \ref{thm:strconvex} \\
& & && \\
\hline
\end{tabular}

\caption{Summary of the complexity results obtained in this paper.}\label{tab:summary}

\end{center}

\end{table}

}

In all cases we assume that $f$ is $L$-smooth (i.e., that the gradient of $f$ is $L$-Lipschitz) and bounded below by $f^*$; the assumptions listed in the first column are {\em additional} to this.   All the results in the table are referring to the setup with $D=D_+ = \{\pm e_i, i=1,2,\dots,n\}$. The general result is obtained by replacing the $n^2$ term by $|D|/\mu^2$, where $\mu$ is the cosine measure of $D$ (however, we conjecture that the ratio $|D|/\mu^2$ can not be smaller than $n^2$ and hence the choice $D=D_+$ is optimal). The quantity $R_0$ measures the size of a specific level set of $f$. Definitions of all quantities appearing in the table are given in the main text.

Notice that the choice of the forcing constant $c$ influences the complexity. It turns out that the choice $c=L/2$ minimizes the complexity bound, which then depends on $L$ linearly. If $c$ is a constant, the dependence becomes quadratic. Hence, the quadratic dependence on $L$ can be interpreted as the price  for not knowing  $L$. 


\textbf{4. Brief and unified analysis, better bounds.} In contrast with existing results, we provide a  {\em brief and unified}  complexity analysis of direct search covering the nonconvex, convex and strongly convex cases.  In particular, the proofs of our complexity theorems covering the nonconvex, convex and strongly convex cases are 6, 10 and 7 lines long, respectively, and follow the same pattern. That is, we show that in all three cases, we have the bound
\[\left\| \nabla f(x_k) \right\| \leq \frac{(\tfrac{L}{2} + c  ) \alpha_0}{ \mu 2^k},\]
where $\{x_k\}$  are the ``unsuccessful'' iterates. We then show that in the convex case this bound implies a bound on $f(x_k)-f(x_*)$ and in the strongly convex case also on $\|x_k-x_*\|$. The difference between the three cases is that the amount of work (function evaluations) needed to reach iterate $k$ differs as we are able to give better bounds in the convex case and even better bounds in the strongly convex case. 

\begin{itemize}
\item \textbf{Nonconvex case.} In the nonconvex case, a relatively brief complexity analysis for direct search was already given by Vicente \cite{vicentenonsmooth}. However, it turns out that the same analysis {\em simplifies substantially} when rewritten to account for  the simplified setting we consider. While it is an easy exercise to check this, this simple observation was not made in the literature before. Moreover,  our proof is different  from this simplified proof.  In more detail, the approach in \cite{vicentenonsmooth} is to  first bound the number of successful steps following the first unsuccessful step, then to bound the unsuccessful steps, and finally to bound the number of steps till the first unsuccessful step. The result is the sum of the three bounds. The complexity theorem assumes that the method converges (and a separate proof is needed for that). In contrast, we do not require a convergence proof to proceed to the complexity proof.  Also, we show that in SDS it is the number of ``unsuccessful'' steps which directly determine the quality of the solution; while the number of ``successful steps'' provides a bound on the workload (i.e., number of function evaluations) needed to find that solution.

\item \textbf{Convex and strongly convex case.} Existing analysis in the convex and strongly convex cases seems to be  more involved and longer \cite[pages 8-18]{vicente} than the analysis in the nonconvex case \cite{vicentenonsmooth}. However, we observe the  analysis and results in \cite{vicente} would  simplify in our simplified setting (i.e., for SDS). Still, the complexity results are weaker than our bounds. For instance, Theorem 4.1 in \cite{vicente}, giving bounds on function values in the convex case, has {\em quadratic dependence} on $R_0$ (we have linear dependence).  We should also remark that in some sense, the setting in \cite{vicente} is already simplified; in particular,  the analysis in  \cite{vicente} only works under the {\em assumption}  that the stepsizes of the  method do not grow above a certain constant $M$.  Note that SDS removes stepsize increases altogether.

\end{itemize}

 \textbf{5. Extensions.} The simplicity of the method and of the analysis makes it easier to propose extensions and improvements. For instance, one may wish to study the complexity of SDS for different function classes (e.g., convex and nonsmooth), different optimization problem (e.g., stochastic optimization, constrained optimization) and  for variants of SDS (e.g., by introducing randomization).

\textbf{6. Wider context.} We now  briefly review some of the vast body of literature on derivative-free optimization. 

It is well known \cite[Section 1.2.3]{nesterovIntro} that  for the problem of unconstrained minimization of a smooth (and not necessarily convex) function, gradient descent takes at most $\mathcal{O}(1/\epsilon^2)$ iterations to drive the norm of the gradient below $\epsilon$. Such a bound has been proved tight in~\cite{cartis}. In the context of derivative-free methods, Nesterov's random Gaussian approach~\cite{nesterovRandomGaussian} attains the complexity bound  $\mathcal{O}(n^2 / \epsilon^2)$. Vicente matches this result with a (deterministic) direct search algorithm~\cite{vicentenonsmooth}, and so does our analysis of direct search. Cartis et al.~\cite{cartisOracle} derived a bound of $\mathcal{O}(n^2 / \epsilon^{3/2})$ for  a variant of their adaptive cubic overestimation algorithm  using finite differences to approximate derivatives. In this setting, Ghadimi and Lan~\cite{lan} achieve better (linear) dependence on $n$ by considering a slightly more special class of problems.

In the convex case, gradient descent achieves the improved bound of $\mathcal{O}(1 / \epsilon)$ \cite[Section 2.1.5]{nesterovIntro}. For derivative-free methods, this rate is also achievable by Nesterov's random Gaussian approach~\cite{nesterovRandomGaussian}  ($O(n/\epsilon)$) and by direct search \cite{vicente} ($O(n^2/\epsilon)$). We match the latter result in this paper.

Optimal (also known as accelerated/fast) gradient methods employ a two step strategy, and enjoy the complexity bound of  $\mathcal{O}(1 / \epsilon^{1/2})$ iterations \cite[Section 2.2.1]{nesterovIntro}. The derivative-free analogue of this method, also developed by Nesterov \cite{nesterovRandomGaussian},  needs $\mathcal{O}(n / \epsilon^{1/2})$ function evaluations. To the best of our knowledge, there are no results on (non-randomized) direct search methods that would attain this bound.

In the strongly convex setting, gradient descent achieves linear convergence, i.e., the bound on number of iterations is $\mathcal{O}(\log(1 / \epsilon))$. This rate is also achievable in derivative-free setting by multiple methods~\cite{vicente, nesterovRandomGaussian, scheinbergBook}, including our version of direct search. 

A recent work of Recht et al. \cite{recht} goes beyond the zero-order oracle. Central in their work is a pairwise comparison oracle, that returns only the order of function values at two different points. They provide lower and upper complexity bounds  for both deterministic and stochastic oracles. A related randomized coordinate descent algorithm is proposed, that also achieves $\mathcal{O} (n\log(1 / \epsilon))$ calls of the oracle for strongly convex functions.  Duchi et al. \cite{duchi} prove tight bounds for online bandit convex optimization problems with multi-point feedback. However, the optimal iteration complexity for single point evaluation still remains an open problem. Yet another related approach, where one has access to partial derivatives, is the randomized coordinate descent method \cite{coordinatedescent}. The iteration complexity of the method is $\mathcal{O}(n / \epsilon)$ in the convex case and $\mathcal{O}(n \log(1 / \epsilon))$ in the strongly convex case. This method can be extended to the derivative-free setting by considering finite difference approximation of partial derivatives.


 \section{Simplified Direct Search} \label{sec:DSO}

In Section~\ref{sec:DSO-methodsimple} we  describe the Simplified Direct Search (SDS) method -- in a user-friendly notation (Algorithm~\ref{alg:dfo-minimal}). Subsequently, in Section~\ref{sec:DSO-method-analysis} we rewrite the method into an analysis-friendly notation (Algorithm~\ref{alg:dfo}) and collect in a lemma a few elementary observations which will be useful later.

\subsection{The method: user-friendly notation} \label{sec:DSO-methodsimple}

SDS (Algorithm~\ref{alg:dfo-minimal}) works with a fixed finite set $D$  of vectors in $\R^n$ forming a positive spanning set (we shall discuss this in detail in Section~\ref{sec:directions}). The mehod is only allowed to take steps of positive lengths, along directions $d\in D$. That is, every update step  is of the form \begin{equation}\label{eq:sjd878dd} x\leftarrow x+ \alpha d,\end{equation} 
for some stepsize $\alpha>0$ and  $d\in D$.  Updates of this form are repeated while they lead to a function decrease of at least $c\alpha^2$ (these steps are called ``successful steps'' and the decrease is called  ``sufficient decrease''), where $c>0$ is a ''forcing constant'' that remains fixed throughout the iterative process. That is, we move from $x$ to $x+\alpha d$ if $d\in D$ is found for which $f(x+\alpha d) \leq f(x)-c\alpha^2$.

\begin{algorithm}[!h]
\vspace{0.5em}
\begin{enumerate}
\item INPUT: starting point $x\in \R^n$; stepsize $\alpha > 0$;  forcing constant $c>0$; finite set $D\subset \R^n$
\item Repeat
\begin{itemize}
\item $\alpha \leftarrow \tfrac{1}{2} \alpha$
\item {\bf while} there exists $d\in D$ such that $f(x+\alpha d) \leq f(x)-c\alpha^2$
\begin{itemize} 
\item[] set $x \leftarrow x+ \alpha d$
\end{itemize}
\end{itemize}
\end{enumerate}
\caption{Simplified Direct Search (SDS): user-friendly form}
\label{alg:dfo-minimal}
\end{algorithm}

There are several ways how, given a stepsize $\alpha$, a vector $d\in D$ leading to sufficient decrease can be found; the method does not prescribe this in more detail. For instance, one may simply order the directions in $D$ and search through them one by one, accepting the {\em first one} that leads to sufficient decrease. Or, one may search through all directions, and if more lead to sufficient decrease,  pick the {\em best one}. Needless to say, the search for a direction can be easily {\em parallelized}. In summary, the method (and our  analysis) is agnostic about the way the selection of direction $d\in D$ leading to sufficient decrease is implemented.

Once no step of the form \eqref{eq:sjd878dd} leads to sufficient decrease, i.e., if $f(x+\alpha d) > f(x)-c\alpha^2$ for all $d\in D$ (we call such steps ``unsuccessful''), we do not\footnote{In practice, one would update $x$ to the best of the points $x+\alpha d$, $d\in D$, if any of them has a smaller function value than $x$. Under such a modification, our theoretical results would either be unchanged (in the convex case: since here we give guarantees in terms of the function value -- which can only get better this way), or would only need minor modifications (in the nonconvex case: here we  give guarantees on the norm of the gradient).} update $x$, halve\footnote{We have chosen, for simplicity of exposition, to present the algorithm and the analysis in the case when the stepsize is divided by 2 at each unsuccessful step. However, one can replace the constant 2 by any other constant larger than 1 and all the results of this paper hold with only minor and straightforward modifications.} the stepsize ($\alpha\leftarrow \alpha/2$) and repeat the process. The rationale behind halving the stepsize also at the beginning of the method will be described in Section~\ref{sec:init}.

Note that the method is  monotonic, i.e., at each successful step the function value decreases, while it stays the same  at unsuccessful steps.

\subsection{The method: analysis-friendly notation} \label{sec:DSO-method-analysis}

As described above, the method is conceptually very simple. However, for the sake of analysis, it will be useful to establish  notation in which we give a special label, $\{x_k\}$, to the unsuccessful iterates $x$, i.e., to points $x$ for which no point of the form $x + \alpha d$, $d\in D$, where $\alpha$ is the current stepsize, leads to  sufficient decrease. The reason for this is the following:  we will prove quality-of-solution guarantees  for the points $x_k$ (in particular, we shall show that $\|\nabla f(x_k)\|=O(1/2^k)$), whereas the number of successful points will determine the associated cost (i.e., number of function evaluations) of achieving this guarantee. 

With this in mind, Algorithm~\ref{alg:dfo-minimal} can be rewritten into an analysis-friendly form, obtaining Algorithm~\ref{alg:dfo}.
For convenience, let us now describe the method again, using the new notation.

\begin{algorithm}
\vspace{0.5em}
\begin{enumerate}
\item INPUT: starting point $x_0\in \R^n$; stepsize $\alpha_0 > 0$; forcing constant $c>0$; finite set $D\subset \R^n$
\item For $k\geq 1$ repeat
\begin{itemize}
\item Set $x_k^0 = x_{k-1}$ and $\alpha_k = \tfrac{1}{2} \alpha_{k-1}$

\item Let $x_k^0, \dots, x_k^{l_k}$ be generated by $$ x_{k}^{l+1} = x_{k}^l + \alpha_k d_{k}^l, \quad  d_k^l \in D, \quad l=0, \dots, l_k-1, $$ so that the following relations hold:
\begin{equation}
f(x_k^{l+1}) \leq f(x_k^l) - c \alpha_k^2, \quad l=0,\dots,l_{k}-1,
\label{eq:js8s}
\end{equation}
and 
\begin{equation}
f(x_k^{l_k} + \alpha_k d) > f(x_k^{l_k}) - c \alpha_k^2 \quad \text{for all} \quad d\in D.
\label{eq:js8s987987}
\end{equation}
\item Set $x_k = x_k^{l_k}$
\end{itemize}
\end{enumerate}
\caption{Simplified Direct Search (SDS): analysis-friendly form}
\label{alg:dfo}
\end{algorithm}

We start\footnote{In Section~\ref{sec:init} we show that it is possible to remove the dependence of the method on $\alpha_0$ or $c$; hence, SDS depends on a single parameter only: (the method still depends on $x_0$ and $D$, but these could simply be set to $x_0=0$ and $D=D_+=\{\pm e_i, \; i=1,2,\dots,n\}$). We prefer the slightly more general form which gives freedom to the user in choosing $x_0$ and $D$.} with an initial iterate $x_0\in \R^n$, an initial stepsize parameter $\alpha_0>0$ and a forcing constant $c>0$.  Given $x_{k-1}$ and $\alpha_{k-1}$, we seek to determine the next iterate $x_k$. This is done as follows. First, we initialize our search for $x_k$ by setting $x_{k}^0 = x_{k-1}$ and decrease the stepsize parameter: $\alpha_{k}=\tfrac{\alpha_{k-1}}{2}$.  Having done that, we try to  find $d\in D$ for which the following sufficient decrease condition holds: 
\[f(x_{k}^0 + \alpha_k d) \leq f(x_k^0) - c\alpha_k^2.\]
If such $d$ exists, we call it $d_k^0$,  declare the search step  {\em successful} and let $x_k^{1} = x_k^0 + \alpha_k d_k^0$. Note that the identification of $x_k^1$ requires, in the worst case, $|D|$ function evaluations (assuming $f(x_0)$ was already computed before). This process is repeated until we are no longer able to find a successful step; that is, until we find $x_k^{l_k}$ which  satisfies \eqref{eq:js8s987987}. Such  a point must exist if we assume that $f$ is bounded below. 

\begin{ass}[Boundedness of $f$] \label{ass:bound} $f$ is bounded below. That is, $f^*\eqdef \inf \{ f(x) \;:\; x\in \R^n\}>-\infty$.
\end{ass}

Indeed,  under Assumption~\ref{ass:bound} it is not possible to keep decreasing the function value by the positive constant $c\alpha_k^2$, and hence $l_k$ must be finite, and $x_k=x_k^{l_k}$ is well defined.  In particular, from \eqref{eq:js8s} we can see that
\[ f(x_k) = f(x_k^{l_k}) \leq f(x_k^{l_k-1}) - c\alpha_k^2 \leq f(x_k^0) - l_k c\alpha_k^2 = f(x_{k-1})-l_k c \alpha_k^2,\]
from which we obtain the following bound:
 \begin{equation}\label{eq:l_k}l_k \leq \frac{f(x_{k-1})-f(x_k)}{c \alpha_k^2} \leq \frac{f(x_0)-f^*}{c\alpha_k^2} = \frac{4^k(f(x_0)-f^*)}{c\alpha_0^2}.\end{equation}
This way, we produce the sequence 
\[x_k^0,x_k^1,\dots, x_k^{l_k},\]
 set $x_{k}=x_{k}^{l_k}$, and proceed to the next iteration. 

Note also that it is possible for $l_k$ to be equal to 0, in which case we have $x_k = x_{k-1}$. However, there is still progress, as the method has learned that the stepsize $\alpha_k$ does not lead to a successful step. 

We now summarize the elementary observations made above.

\begin{lem} \label{lem:firstobserv}Let Assumption~\ref{ass:bound} be satisfied. Then
\begin{itemize} 
\item[(i)] The counters $l_k$ are all finite, bounded as in \eqref{eq:l_k}, and hence the method produces an infinite sequence of iterates $\{x_k\}_{k\geq 0}$ with non-increasing function values: $f(x_{k+1})\leq f(x_k)$ for all $k\geq 0$.  
\item[(ii)] The total number of function evaluations up to iteration $k$ is bounded above by
\begin{equation}  \label{eq:js9jd8dd} N(k) \eqdef 1+ \sum_{j=1}^k |D|(l_j+1).\end{equation}
\item[(iii)] For all $k\geq 1$ we have
\begin{equation}\label{eq:s98yushus}f(x_k+ \alpha_k d) > f(x_{k}) - c \alpha_k^2 \qquad \text{for all} \qquad d\in D.\end{equation}
\end{itemize}
\end{lem}
\begin{proof}
We have established part (i) in the preceding text. The leading ``1'' in \eqref{eq:js9jd8dd} is for the evaluation of $f(x_0)$. Having computed $f(x_{j-1})$,  the method needs to perform at most $|D| (l_j +1)$ function evaluations  to identify $x_j$: up to $|D|$ evaluations to identify each of the points $x_j^{1}, \dots, x_j^{l_j}$ and further $|D|$ evaluations to verify that \eqref{eq:js8s987987} holds. It remains to add these up; which gives (ii). Part (iii) follows trivially by inspecting  \eqref{eq:js8s987987}.
 \end{proof}

%
%

\section{Initialization}\label{sec:init}

It will be convenient  to assume---without loss of generality as we shall see---that \eqref{eq:s98yushus} holds for $k=0$ also. In fact, this is the reason for halving the initial stepsize at the {\em beginning} of the iterative process. Let us formalize this as an assumption:

\begin{ass}[Initialization] \label{ass:init} The triple $(x_0,\alpha_0,c)$ satisfies the following relation:
\begin{equation}
\label{eq:hss8jss} f(x_0 + \alpha_0 d) > f(x_0) - c \alpha_0^2 
\qquad  \text{for all} \qquad  d \in D. \end{equation}
\end{ass}

In this section we describe three ways of identifying a triple $(x_0,\alpha_0,c)$ for which Assumption~\ref{ass:init} is satisfied; some more efficient than others. 
\begin{itemize}
\item {\em Bootstrapping initialization.} Finds a suitable starting point $x_0$.
\item {\em Stepsize initialization.} Finds a ``large enough'' stepsize $\alpha_0$.
\item {\em Forcing constant initialization.} Finds a ``large enough'' forcing constant $c$.
\end{itemize}

 We will now show that initialization can be performed very efficiently (in particular, we prefer {\em stepsize initialization} and forcing constant initialization to bootstrapping), and hence Assumption~\ref{ass:init} does not pose any practical issues. In fact, we found that stepsize initialization has the capacity to dramatically improve the practical performance of SDS, especially when compared to the bootstrapping strategy --- which is the strategy {\em implicitly} employed by  direct search methods in the literature \cite{vicentenonsmooth, vicente}. Indeed, bootstrapping initialization is equivalent to running SDS without any initialization whatsoever.

\subsection{Bootstrapping initialization}
 
A natural way to initialize---although  this may be very inefficient in both theory and practice---is to simply run the direct search method itself (that is,  Algorithm~\ref{alg:dfo-minimal}; without halving the stepsize), using any triple $(\tilde{x}_0,\alpha_0,c)$, and stop as soon as the first unsuccessful iterate $x$ is found.  We can then set $x_0\leftarrow x$ and keep $\alpha_0$ and $c$. This method is formally described in Algorithm~\ref{alg:init-bootstrap}.

\begin{algorithm}[!h]
\vspace{0.5em}
\begin{enumerate}
\item INPUT: $\tilde{x}_0 \in \R^n$; stepsize $\alpha_0 > 0$; forcing constant $c>0$; $D\subset \R^n$
\item $x\leftarrow \tilde{x}_0$
\item {\bf while} there exists $d\in D$ such that $f(x+\alpha_0 d) \leq f(x)-c\alpha_0^2$
\begin{itemize}
\item  set $x \leftarrow x+ \alpha_0d$
\end{itemize}

\item  OUTPUT: $x_0 = x$, $\alpha_0 = \alpha_0$ and $c=c$
\end{enumerate}
\caption{Bootstrapping initialization (finding suitable $x_0$)}
\label{alg:init-bootstrap}
\end{algorithm}

\begin{lem} Let Assumption~\ref{ass:bound} ($f$ is bounded below) be satisfied. Then Algorithm~\ref{alg:init-bootstrap} outputs triple  $(x_0,\alpha_0,c)$ satisfying Assumption~\ref{ass:init}. Its complexity is
\begin{equation} \label{eq:iuhisuhsi} I_{x_0} \eqdef  |D| \left( \frac{f(\tilde{x}_0)-f^*}{c\alpha_0^2} +1\right) \end{equation}
function evaluations (not counting the evaluation of $f(x_0)$). .
\end{lem}
\begin{proof}
Starting from $f(\tilde{x}_0)$, the function value cannot decrease by more than $f(\tilde{x}_0)-f^*$, hence, $x$ is updated at most $(f(\tilde{x}_0)-f^*)/(c\alpha_0^2)$ times. During each such decrease at most $|D|$ function evaluations are made. There is an additional $|D|$ term to account for checking at the end that \eqref{eq:hss8jss} holds. [Note that we are using the same logic which led to bound \eqref{eq:l_k}.]
\end{proof}

\subsection{Stepsize initialization}
 
We now propose a  much more efficient initialization procedure (Algorithm~\ref{alg:alpha0}).  Starting with $(x_0, \tilde{\alpha}_0,c)$,  the procedure finds large enough  $\alpha_0$  (in particular, $\alpha_0\geq \tilde{\alpha}_0$), so that the triple $(x_0,\alpha_0,c)$ satisfies Assumption~\ref{ass:init} -- if $f$ is {\em convex}.    It turns out, as we shall see, that the initialization step is not needed in the nonconvex case, and hence this is sufficient. At the same time, we would like to be able to use this initialization algorithm also in the nonconvex case -- simply because in derivative-free optimization we may not {\em know} whether the function we are minimizing is or is not convex!  So, in this case we would like to be  at least able to say that the initialization algorithm stops after a certain (small) number of function evaluations, and be able to say something about how large $\alpha_0$ is.

\begin{algorithm}
\begin{enumerate}
\vspace{0.5em}
\item INPUT: $x_0\in \R^n$; stepsize $\tilde{\alpha}_0>0$;   forcing constant $c>0$; $D = \{d_1, d_2, \dots, d_p\}$
\item $i \leftarrow 1$ and $\alpha \leftarrow \tilde{\alpha}$
\item \textbf{while} $i \leq |D|$
\begin{itemize}
  \item \textbf{if} $f(x_0 + \alpha d_i) \leq f(x_0) - c \alpha^2$ \textbf{then} set $\alpha \leftarrow 2 \alpha$
\item \textbf{else}  $i \leftarrow i+1$
\end{itemize}
\item OUTPUT: $\alpha_0 = \alpha$
\end{enumerate}
\caption{Stepsize initialization (finding large enough $\alpha_0$)}
\label{alg:alpha0}
\end{algorithm}

The following result describes the behaviour of the stepsize initialization method.

\begin{lem}\label{thm:init} Let Assumption~\ref{ass:bound} ($f$ is bounded below) be satisfied. 
\begin{enumerate}
\item[(i)] Algorithm~\ref{alg:alpha0} outputs $\alpha_0$ satisfying 
\[1 \leq \frac{\alpha_0}{\tilde{\alpha}_0} \leq \max \left\{1, 2\sqrt{\frac{f(x_0)-f^*}{c\tilde{\alpha}_0^2}} \right\} \eqdef M,\]
and performs in total at most
\begin{equation}\label{eq:jsuhujns}  I_{\alpha_0} \eqdef  |D| + \log_2 M  =|D| + \max\left \{0, 1 + \frac{1}{2}\log_2 \left(\frac{f(x_0)-f^*}{c\tilde{\alpha}_0^2}\right) \right\} \end{equation}
function evaluations (not counting the evaluation of $f(x_0)$). 
\item[(ii)] If $f$ is convex, then the triple $(x_0,\alpha_0,c)$ satisfies Assumption \ref{ass:init}.
\end{enumerate}
\end{lem}

\begin{proof}
(i) If  at some point during the execution of the algorithm we have $\alpha > \sqrt{(f(x_0)-f^*)/c} \eqdef h$, then the ``if'' condition cannot be satisfied, and hence $\alpha$ will not be further doubled. So, if $\tilde{\alpha} \leq h$, then   $\alpha \leq 2h$ for all $\alpha$ generated throughout the algorithm, and if $\tilde{\alpha} > h$, then $\alpha=\tilde{\alpha}$ throughout the algorithm. Consequently, $\alpha\leq \max\{\tilde{\alpha},2h\}$ throughout the algorithm.
 
Now note that at each step of the method, either $\alpha$ is doubled, or $i$ is increased by one. Since, as we have just shown, $\alpha$ remains bounded, and $D$ is finite, the algorithm will stop. Moreover, the method performs function evaluation of the form $f(x_0+\alpha d_i)$, where $\alpha$ can assume at most $1+\log_2 M$ different values and $d_i$ at most $|D|$ different values, in a fixed order. Hence, the method performs at most $ |D| + \log_2 M$ function evaluations (not counting $f(x_0)$).

(ii)  Note that for each $d_i \in D$ there exists $\alpha_i \leq \alpha_0$ for which \begin{equation}\label{eq:js8ndgdd}f(x_0 + \alpha_i d_i) > f(x_0) - c \alpha_i^2.\end{equation} Indeed, this holds for $\alpha_i$ equal to the value of $\alpha$ at the moment when the index $i$ is increased. We now claim that, necessarily,  
inequality \eqref{eq:js8ndgdd} must hold with $\alpha_i$ replaced by $\alpha_0$. We shall show that this follows from convexity. Indeed, by convexity,
\[\frac{f(x_0+\alpha_0d_i)-f(x_0)}{\alpha_0} \geq \frac{f(x_0+\alpha_i d_i )- f(x_0)}{\alpha_i},\]
which implies that 
$$ f(x_0 + \alpha_0 d_i) \geq f(x_0) + \left( f(x_0 + \alpha_i d_i) - f(x_0) \right) \frac{\alpha_0}{\alpha_i} > f(x_0) - c \alpha_i^2 \frac{\alpha_0}{\alpha_i} \geq f(x_0) - c \alpha_0^2. $$
We have now established \eqref{eq:hss8jss}, and hence the second statement of the theorem is proved.

\end{proof}

The runtime of the stepsize initialization method (Algorithm~\ref{alg:alpha0}), given by \eqref{eq:jsuhujns}, is negligible compared to the runtime of the bootstrapping initialization method (Algorithm~\ref{alg:init-bootstrap}). 


%

\subsection{Forcing constant initialization} 

In Algorithm~\ref{alg:init-forcing} we  set $c$ to a large enough value, given  $x_0$ and $\alpha_0$.
 
\begin{algorithm}[!h]
\vspace{0.5em}
\begin{enumerate}
\item INPUT: $x_0 \in \R^n$; stepsize $\alpha_0 > 0$;  $D\subset \R^n$
\item OUTPUT: $x_0=x_0$, $\alpha_0=\alpha_0$ and $c = 1+ \max\left\{0, \tfrac{f(x_0)-\min_{d\in D} f(x_0 + \alpha_0 d)}{\alpha_0^2}\right\}$
\end{enumerate}
\caption{Forcing constant initialization (finding suitable $c$)}
\label{alg:init-forcing}
\end{algorithm}

\begin{lem}  Algorithm~\ref{alg:init-forcing} outputs triple  $(x_0,\alpha_0,c)$ satisfying Assumption~\ref{ass:init}. Its complexity is
\begin{equation} \label{eq:iuhisuhsi} I_{c} \eqdef |D| \end{equation}
function evaluations (not counting the evaluation of $f(x_0)$). 
\end{lem} 
\begin{proof}
By construction,  $c$ is positive and $c\alpha_0^2 > f(x_0) - f(x_0 + \alpha_0 d)$ for all $d\in D$. The method needs to evaluate $f(x_0+\alpha_0 d)$ for all $d\in D$.
\end{proof}

\section{Positive spanning sets and their cosine measure} \label{sec:directions}

Clearly, the set of directions, $D$, needs to be rich enough so that every point in $\R^n$ (in particular, the optimal point) is potentially reachable by a sequence of steps of SDS. In particular, we will assume that $D$ is a {\em positive spanning set}; that is, the conic hull of $D$ is $\R^n$: \[\R^n = \left\{\sum_{i} t_i d_i, \; d_i\in D,\; t_i\geq 0 \right\}.\]
Proposition~\ref{prop:D} lists several equivalent characterizations of a positive spanning set. We do not need the result to prove our complexity bounds; we include it for the benefit of the reader.A proof sketch can be found in the appendix.

\begin{prop}\label{prop:D} Let $D$ be a finite set of nonzero vectors. The following statements are equivalent:
\begin{itemize}
\item[(i)] $D$ is a positive spanning set.
\item[(ii)] The cosine measure of $D$, defined below\footnote{Note that the continuous function $v\mapsto \max_{d\in D} \ve{v}{d}/\|d\|$ attains its minimizer on the compact set $\{v\;:\; \|v\|=1\}$. Hence, it is justified to write minimum in \eqref{eq:hjs787gasas} instead of infimum.}, is positive:
\begin{equation}\label{eq:hjs787gasas} \mu(D) = \mu \eqdef \min_{0\neq v\in \R^n} \max_{d\in D} \frac{\ve{v}{d}}{\|v\|\|d\|} > 0.\end{equation}
Above, $\ve{\cdot}{\cdot}$ is the standard Euclidean inner product and $\|\cdot\|$ is the standard Euclidean norm. 
\item[(iii)] The convex hull of $D$ contains $0$ in its interior\footnote{It is clear from the proof provided in the appendix that if (ii) holds, then the convex hull of $\{d/\|d\|\;\:\; d\in D\}$ contains the ball centered at the origin  of radius $\mu$.}.
\end{itemize}
\end{prop}

This is formalized in the following assumption.

\begin{ass}[Positive spanning set]  \label{ass:D} $D$ is a finite set of unit-norm in $\R^n$ forming a positive spanning set. 
\end{ass}

Note that we  assume that all vectors in $D$ are of unit length. While the algorithm and theory can be extended in a straightforward way to allow for vectors of different lengths (which, in fact, is standard in the literature), this does not lead to an improvement in the complexity bounds and merely makes the analysis and results a bit less transparent. Hence, the unit length assumption is enforced for convenience.

%

 This assumption is standard in the literature on direct search. Indeed, it is clearly  necessary  as otherwise it is not possible to guarantee that any point (and, in particular, the optimal point) can be reached by a sequence of steps of the algorithm.

The cosine measure $\mu$ has  a straightforward geometric interpretation: for each nonzero vector $v$, let $d\in D$ be the vector forming the smallest angle with $v$ and let $\mu(v)$ be the cosine of this angle. Then $\mu = \min_v \mu(v)$. That is, for every nonzero vector $v$ there exists $d\in D$ such that the cosine of the angle between these two vectors is at least $\mu>0$ (i.e., the angle is acute). In the analysis, we shall consider the vector $v$ to be the negative gradient of $f$ at the current point. While this gradient is unknown, we know that there is a direction in $D$ which approximates it well, with the size of $\mu$ being a measure of the quality of that approximation: the larger $\mu$ is, the better.  

Equivalently, $\mu$ can be seen as the largest scalar such that for all nonzero $v$ there exists $d\in D$ so that the following inequality holds:
\begin{equation}\label{eq:js8js8s}\mu \|v\|\|d\| \leq \ve{v}{d}.\end{equation}
This is a reverse of the Cauchy-Schwarz inequality, and hence, necessarily, $\mu \leq 1$. However, for $\mu=1$ to hold we would need $D$ to be dense on the unit sphere. For better insight, consider the following example. If $D$ is chosen to be the ``maximal positive basis" (composed of the coordinate vectors together with their negatives: $D = \{ \pm e_i \ |\ i = 1, \dots, n \}$), then 
\begin{equation}
\label{eq:muexample}
\mu = \frac{1}{\sqrt{n}}.
\end{equation}

\section{Complexity analysis}\label{sec:main_results}

In this section we state and prove our main results: three complexity theorems covering the nonconvex, convex and strongly convex case. We also provide a brief discussion.

In all results of this section we will make the following assumption.

\begin{ass}[$L$-smoothness of $f$] \label{ass:smooth} $f$ is $L$-smooth. That is, $f$ has a Lipschitz continuous gradient, with a positive Lipschitz constant $L$: 
\begin{equation}
\label{eq:Lipschitz}
\| \nabla f(x) - \nabla f(y) \| \leq L \| x - y \| \quad \text{for all} \quad x,y\in \R^n.
\end{equation}
\end{ass}

\subsection{Key lemma}\label{ass:lemmas}

In this section we establish a key result that will drive the analysis.
The result is standard in the analysis of direct search methods \cite{scheinbergBook}, although we only need it in a simplified form\footnote{Lemma~\ref{lem1} is usually stated in a setting with the vectors in $D$ allowed to be of arbitrary lengths, and with $c\alpha^2$ replaced by an arbitrary {\em forcing function} $\rho(\alpha)$. In this paper we chose to present the result with $\rho(\alpha)=c\alpha^2$ since i) the complexity guarantees do not improve by considering a different forcing function, and because ii) the results and proofs become a bit less transparent. For a general forcing function, the statement would say that if $f(x) - \rho(\alpha) < f(x + \alpha d)$ for all $d \in D$, then
\[\|\nabla f(x)\| \leq \mu^{-1} \left(\tfrac{L}{2} \alpha d_{max} + \tfrac{\rho(\alpha)}{\alpha } \tfrac{1}{d_{min}} \right),\]
where $d_{min} = \min \{\|d\| \;:\; d \in D\}$ and $d_{max}=\max \{\|d\|\;:\; d\in D\}$. In this general form the lemma is presented, for instance, in  \cite{scheinbergBook}.}.
Moreover, we will use it in a  novel way, which leads to a significant simplification (especially in the convex cases) and unification of the analysis, and to sharper and cleaner complexity bounds. We include the proof for completeness.


\begin{lem}
\label{lem1} Let Assumption \ref{ass:smooth} ($f$ is $L$-smooth) and Assumption \ref{ass:D} ($D$ is a positive spanning set) be satisfied and let $x \in \R^n$ and $\alpha > 0$. If $f(x) - c \alpha^2 < f(x + \alpha d)$ for all $d \in D$, then
\begin{equation}
\label{eq:lem1}
\| \nabla f(x) \| \leq \frac{1}{\mu} \left(\frac{L}{2} + c \right) \alpha.
\end{equation}

\end{lem}

\begin{proof}
Since $f$ is $L$-smooth, \eqref{eq:Lipschitz} implies that for all $x, y \in \R^n$ we have $ f(y) \leq f(x) + \< \nabla f(x), y - x > + \frac{L}{2} \| y - x \|^2$. By assumption, we know that for all $d\in D$, we have $- f(x + \alpha d) < - f(x) + c \alpha^2$. Summing these two  inequalities, and setting $y = x + \alpha d$, we obtain \begin{equation}\label{eq:shs68s9s} 0 < \< \nabla f(x), \alpha d > + c \alpha^2 + \frac{L}{2} \| \alpha d \|^2. \end{equation}
Let $d \in D$ be such that $\mu \| \nabla f(x) \| \| d \| = \mu \| \nabla f(x) \| \leq -\ve{\nabla f(x)}{d}$ (see \eqref{eq:js8js8s}). Inequality \eqref{eq:lem1} follows by multiplying this inequality by $\alpha$, adding it to \eqref{eq:shs68s9s} and rearranging the result. 
\end{proof}

\subsection{Nonconvex case}\label{sec:2nonconvex} 

In this section, we state our most general complexity result -- one that does not require any additional assumptions on $f$, besides smoothness and boundedness. In particular, it applies to non-convex objective functions.

\begin{theorem}[Nonconvex case] Let Assumptions  \ref{ass:smooth} ($f$ is $L$-smooth) and \ref{ass:D} ($D$ is a positive spanning set) be satisfied. Choose initial iterate $x_0\in \R^n$ and initial stepsize parameter $\alpha_0>0$. Then the iterates $k\geq 1$ of Algorithm~\ref{alg:dfo} satisfy:
\begin{equation}\label{eq:iuh8hd}
\|\nabla f(x_k)\| \leq  \frac{\left( \frac{L}{2}+c\right) \alpha_0}{\mu 2^k}.
\end{equation}
Pick any $0<\epsilon<(\tfrac{L}{2}+c)\tfrac{\alpha_0}{\mu}$. If, morevoer,  Assumption~\ref{ass:bound}  ($f$ is bounded below) is satisfied and we set  \begin{equation}
\label{eq:sjs8jss}k\geq k(\epsilon) \eqdef \left \lceil \log_2 \left( \frac{ (\tfrac{L}{2}+ c)\alpha_0}{\mu \epsilon} \right)  \right \rceil,\end{equation} then $\left\| \nabla f(x_{k(\epsilon)}) \right\| \leq \epsilon$, while the method performs in total at most 
\begin{equation} 
N(k(\epsilon)) \leq 1 + |D| \left( k(\epsilon) +   \frac{16(f(x_0) - f^*)(\frac{L}{2}+c)^2 }{3c \mu^2 \epsilon^2}\right) 
\label{eq:nonconvexBound}
\end{equation}
function evaluations. 
\label{thm:nonconvex}
\end{theorem}

\begin{proof}
Inequality \eqref{eq:iuh8hd} follows from \eqref{eq:lem1} by construction of $x_k$ (see  \eqref{eq:js8s987987}) and $\alpha_k$. Since $k(\epsilon)\geq 1$, 
\[\|\nabla f(x_{k(\epsilon)})\| \overset{\eqref{eq:iuh8hd}}{\leq}  \frac{\left(\frac{L}{2}+c\right) \alpha_0}{\mu 2^{k(\epsilon)}}  \overset{\eqref{eq:sjs8jss}}{\leq } \epsilon.\] 

By substituting   the second estimate in \eqref{eq:l_k} into \eqref{eq:js9jd8dd} and using the fact that $\alpha_k=\alpha_0/ 2^k$, $k\geq 0$, we obtain the bound
\begin{equation}\label{eq:jsihsh} N(k) \leq 1 + k|D| + \frac{4(4^{k}-1)|D|}{3c \alpha_0^2}(f(x_0)-f^*),\end{equation}
from which with $k=k(\epsilon)$ we get \eqref{eq:nonconvexBound}.
\end{proof}

We shall now briefly comment the above result.
\begin{itemize}
\item Notice that we do not enforce Assumption~\ref{ass:init} -- no initialization is needed if $f$ is not convex. This is because in the nonconvex case the best bound on $l_1$ is the one we have used in the analysis (given by \eqref{eq:l_k}) -- and it is not improved by enforcing Assumption~\ref{ass:init}. As we shall see, in the convex and strongly convex cases a better bound on $l_1$ is available if we  enforce Assumption~\ref{ass:init} (and, as we have seen in Section~\ref{sec:init}, initialization is cheap), which leads to better  complexity.
\item In the algorithm we have freedom in choosing $c$. It is easy to see that the choice $c = \tfrac{L}{2}$ minimizes the dominant term in the complexity bound \eqref{eq:nonconvexBound}, in which case the bound takes the form
\begin{equation} 
\mathcal{O} \left( \frac{|D|}{\mu^2}\frac{ L (f(x_0)-f^*)}{ \epsilon^2} \right).\label{eq:nonconvexCBound} 
\end{equation}
Needless to say, in a derivative-free setting the value of $L$ is usually not available and hence usually one cannot choose $c=\tfrac{L}{2}$. For $c=O(1)$, the complexity depends quadraticaly  on $L$. 

\item If $D$ is chosen to be the ``maximal positive basis" (see \eqref{eq:muexample}), the bound \eqref{eq:nonconvexCBound} reduces to
$$ \mathcal{O} \left( \frac{n^2 L (f(x_0)-f^*)}{\epsilon^2} \right). $$
This form of the result was used in Table~\ref{tab:summary}.
\end{itemize}

\subsection{Convex case} \label{sec:2convex} 

In this section, we analyze the method under the additional assumption that $f$ is convex. For technical reasons, we also assume that the problem is solvable (i.e., that it has a minimizer $x_*$) and that, given an initial iterate $x_0 \in \R^n$, the quantity
\begin{equation} 
\label{eq:R}
R_0 \eqdef  \sup_{x \in \R^n} \{ \| x - x_* \| \;: \; f(x) \leq f(x_0) \}
\end{equation}
is finite. Further, for convenience we define
\begin{equation} 
\label{eq:B}
B \eqdef \frac{R_0 (\frac{L}{2}+c)}{\mu}.
\end{equation}

We are now ready to state the complexity result.
\begin{theorem}[Convex case] 
\label{thm:convex}
Let Assumptions \ref{ass:smooth} ($f$ is $L$-smooth) and \ref{ass:D} ($D$ is a positive spanning set) be satisfied. Further assume that $f$ is convex, has a minimizer $x_*$ and  $R_0<\infty$ for some initial iterate  $x_0\in \R^n$. Finally, let Assumption~\ref{ass:init} (initialization) be satisfied.  Then:
\begin{itemize}
\item [(i)] The iterates $\{x_k\}_{k\geq 0}$ of Algorithm~\ref{alg:dfo} satisfy 
\begin{equation} \label{eq:shsts}\left\| \nabla f(x_k) \right\| \leq \frac{(\tfrac{L}{2} + c  ) \alpha_0}{ \mu 2^k}, \quad f(x_k) - f(x_*)\leq  \frac{B \alpha_0}{2^k}, \quad k \geq 0, \end{equation}
where at iteration $k$ the method needs to perform at most $|D| \left( 1 + \frac{2^{k+1} B}{c \alpha_0} \right)$ function evaluations.
\item[(ii)]
In particular, if we set $k = k(\epsilon) \eqdef \left\lceil \log_2 \left( \frac{B \alpha_0}{\epsilon} \right) \right\rceil$, where $0<\epsilon\leq B\alpha_0$, then  $f(x_k) - f(x_*) \leq \epsilon$, while Algorithm~\ref{alg:dfo} performs in total  at most
\begin{equation} 
\label{eq:complexity-convex}
N(k(\epsilon)) \leq 1+ |D| \left( k(\epsilon) + \frac{8 B^2}{c \epsilon} \right)
\end{equation}
function evaluations.
\end{itemize}
\end{theorem}
\begin{proof}
The first part of \eqref{eq:shsts}, for $k\geq 1$, follows from Theorem~\ref{thm:nonconvex}. For $k=0$ it follows by combining Assumption~\ref{ass:init} and Lemma~\ref{lem1}. In order to establish the second part of \eqref{eq:shsts},  
it suffices to note that  since $f(x_k)\leq f(x_0)$ for all $k$ (see Lemma~\ref{lem:firstobserv}(i)), we have for all $k\geq 0$:
\begin{equation}\label{eq:siu87gdd}f(x_k) - f(x_*) \leq \langle \nabla f(x_k), x_k - x_* \rangle \leq \| \nabla f(x_k) \| \| x_k - x_* \| \overset{ \eqref{eq:R} }{\leq}  \|\nabla f(x_k)\| R_0 \overset{\eqref{eq:shsts}+ \eqref{eq:B}}{\leq} B \alpha_k.\end{equation} 
It only remains to establish \eqref{eq:complexity-convex}. Letting $r_k = f(x_k)-f^*$, and using \eqref{eq:js8s} and \eqref{eq:siu87gdd},  we have $0\leq r_k\leq r_{k-1}  - l_k c \alpha_k^2 \leq B \alpha_{k-1} - l_k c \alpha_k^2 = 2B \alpha_{k} - l_k c \alpha_k^2 $, whence  \begin{equation} \label{eq:jsuhs}l_k \leq \frac{2B}{c \alpha_k} = \frac{2^{k+1}B}{c \alpha_0}.\end{equation}
We can now estimate the total number of function evaluations by plugging \eqref{eq:jsuhs} into   \eqref{eq:js9jd8dd}:
\begin{eqnarray*}
N(k(\epsilon)) &\overset{\eqref{eq:js9jd8dd}}{=}  & 1+ \sum_{k=1}^{k(\epsilon)}|D| (l_k+1) \;\; \overset{\eqref{eq:jsuhs}}{\leq} \;\; 1+  |D| \sum_{k=1}^{k(\epsilon)} \left( \frac{2B}{c \alpha_k} +1\right) \;\; =\;\; 1 + |D| k(\epsilon) + \frac{2B |D|}{c \alpha_0} \sum_{k=1}^{k(\epsilon)} 2^k \\
& \leq & 1 + |D| k(\epsilon) + \frac{2B|D|}{c\alpha_0} 2^{k(\epsilon)+1} \;\; \leq \;\; 1+  |D| \left( k(\epsilon) + \frac{8 B^2}{c \epsilon} \right).
\end{eqnarray*} 
\end{proof}

We shall now comment the result.
\begin{itemize}
\item Note that Assumption~\ref{ass:init} was used to argue that $r_0\leq B \alpha_0$, which was in turn used to bound $l_1$. The  bounds on $l_k$ for $k>1$ hold even without this assumption. Alternatively, we could have skipped Assumption~\ref{ass:init} and bounded $l_1$ as in Theorem~\ref{thm:nonconvex}. Relations \eqref{eq:shsts} would hold for $k\geq 1 $.

\item Again, we have freedom in choosing $c$ (and  note that $c$ appears also in the definition of $B$). It is easy to see that the choice $c=\tfrac{L}{2}$ minimizes the dominating term $\frac{B^2}{c}$ in the complexity bound \eqref{eq:complexity-convex}, in which case $B=\tfrac{LR_0}{\mu}$, $\tfrac{B^2}{c}=\tfrac{2LR_0^2}{\mu^2}$ and the  bound \eqref{eq:complexity-convex} takes the form
\begin{equation} 
1+|D| \left[ \left\lceil \log_2 \left(\frac{L R_0 \alpha_0}{\mu \epsilon} \right) \right\rceil + \frac{16 L R_0^2}{\mu^2 \epsilon} \right] = {\cal O}\left(\frac{|D|}{\mu^2} \frac{L R_0^2 }{ \epsilon} \right).\label{eq:bound 2} 
\end{equation}

\item If $D$ is chosen to be the ``maximal positive basis" (see \eqref{eq:muexample}), the bound \eqref{eq:bound 2} reduces to
\[{\cal O}\left(\frac{n^2 L R_0^2 }{\epsilon}\right).\]
The result is listed in this form in Table~\ref{tab:summary}.

\item It is possible to improve the algorithm by introducing an additional stopping criterion: $l_k \geq \tfrac{B}{c \alpha_k}$. The analysis is almost the same, and the resulting number of function evaluations is halved in this case. However, this improvement is rather theoretical, since we typically do not know the value of $B$.
\end{itemize}

\subsection{Strongly convex case}\label{sec:2strconvex} 

In this section we introduce an additional assumption:  $f$ is $\lambda$-strongly convex for some (strong convexity) constant $\lambda >0$. That is, we require that $\forall x, y \in \R^n$, we have
\begin{equation}
\label{eq:strconv}
f(y) \geq f(x) + \< \nabla f(x), y - x > + \frac{\lambda}{2} \|  y - x \|^2.
\end{equation}
In particular, by minimizing both sides of the above inequality in $y$, we obtain the standard inequality
\begin{equation}
\label{eq:stroconvexproperty}
f(x) - f(x_*) \leq \frac{1}{2\lambda} \| \nabla f(x) \|^2.
\end{equation}
Moreover, by substituting  $y \leftarrow x$ and $x \leftarrow x_*$ into \eqref{eq:strconv}, and using the fact that $\nabla f(x_*)=0$, we obtain another well known inequality:
\begin{equation}\label{eq:js8gs098}
\frac{\lambda}{2}\|x-x_*\|^2 \leq f(x)-f(x_*).
\end{equation}
To simplify  notation, in what follows we will make use of the following quantity:
\begin{equation}
\label{eq:Bstrcvx}
S \eqdef \frac{( \tfrac{L}{2} +c)^2}{2 \lambda \mu^2}.
\end{equation}

\begin{theorem} [Strongly convex case]
\label{thm:strconvex}
Let Assumptions \ref{ass:smooth} ($f$ is $L$-smooth), \ref{ass:D} ($D$ is a positive spanning set) and \ref{ass:init} (initialization) be satisfied.
Further, assume that $f$ is $\lambda$-strongly convex. Then:
\begin{itemize}
\item[(i) ] The iterates $\{x_k\}_{k\geq 0}$ of Algorithm~\ref{alg:dfo} satisfy 
\begin{equation}\label{eq:hjs87ysb} \left\| \nabla f(x_k) \right\| \leq \frac{(\tfrac{L}{2} + c) \alpha_0}{2^k \mu}, \quad f(x_k) - f(x_*) \leq S \left( \frac{\alpha_0}{2^k} \right)^2, \quad  \|x_k-x_*\| \leq \frac{(\tfrac{L}{2} + c) \alpha_0}{2^k \mu \lambda}, \end{equation}
where at each iteration the method needs to perform at most $|D|\left(  \frac{4S}{c} +1\right)$ function evaluations.

\item[(ii)] In particular, if we set $ k = k(\epsilon) \eqdef \left\lceil \log_2 \left( \alpha_0 \sqrt{\frac{S}{\epsilon}} \right) \right\rceil$, where $0<\epsilon< S\alpha_0^2$, then  $f(x_k) - f(x_*) \leq \epsilon$, while the method performs in total at most
\begin{equation}
N(k(\epsilon)) \leq 1 + |D| \left(  \frac{4S}{c} + 1\right) \left( 1 + \log_2 \left( \alpha_0 \sqrt{\frac{S}{\epsilon}} \right) \right)
\label{eq:sjs8sjs}
\end{equation}
function evaluations.
\end{itemize}
\end{theorem}
\begin{proof}
The first part of \eqref{eq:hjs87ysb} was already proved in the convex case \eqref{eq:shsts}; the rest follows from
\[ f(x_k)-f(x_*) \overset{\eqref{eq:stroconvexproperty}}{\leq} \tfrac{1}{2\lambda}\| \nabla f(x_k) \|^2 \overset{\eqref{eq:hjs87ysb} + \eqref{eq:Bstrcvx} }{\leq} S \alpha_k^2;\quad 
\|x_k-x_*\| \overset{\eqref{eq:js8gs098}+\eqref{eq:stroconvexproperty}}{\leq} \frac{\|\nabla f(x_k)\|}{\lambda} \overset{\eqref{eq:hjs87ysb}}{\leq} \frac{(\tfrac{L}{2}+c) \alpha_0}{2^k \mu \lambda}.\]

It only remains to establish the bound \eqref{eq:sjs8sjs}. Letting $r_k = f(x_k)-f^*$,  and using \eqref{eq:js8s} and the second inequality in \eqref{eq:hjs87ysb}, we have $0\leq r_k\leq r_{k-1}  - l_k c \alpha_k^2 \leq S (2 \alpha_{k})^2 - l_k c \alpha_k^2$, whence  \begin{equation} \label{eq:jsuhs}l_k \leq \frac{4S\alpha_k^2}{c \alpha_k^2} = \frac{4S}{c }.\end{equation}
We can now estimate the total number of function evaluations by plugging \eqref{eq:jsuhs} into  \eqref{eq:js9jd8dd}:
\begin{eqnarray*}N(k(\epsilon)) &\overset{\eqref{eq:js9jd8dd}}{=}  & 1+ \sum_{k=1}^{k(\epsilon)}|D| (l_k+1) \;\; \overset{\eqref{eq:jsuhs}}{\leq}\;\; 1+ \sum_{k=1}^{k(\epsilon)}|D| \left(\frac{4S}{c} +1\right) \;\; =\;\; 1 + |D| \left(\frac{4S}{c} +1\right) k(\epsilon).\end{eqnarray*}
\end{proof}

Let us now comment on the  result.

\begin{itemize}

\item As before, in the algorithm we have freedom in choosing $c$. Choosing $c = \tfrac{L}{2}$ minimizes the dominating term $\frac{S}{c}$ in the complexity bound~\eqref{eq:sjs8sjs}, in which case $S = \tfrac{L^2}{2\lambda \mu^2}$, $\tfrac{S}{c} = \tfrac{L}{\lambda \mu^2}$ and the bound~\eqref{eq:sjs8sjs} takes the form
\begin{equation} 
\label{eq:098jsusjs} 
1 + |D| \left(1 + \frac{4L}{\lambda \mu^2} \right) \left( 1 + \log_2 \left( \frac{\alpha_0 L}{\mu} \sqrt{\frac{1}{2 \lambda \epsilon}} \right) \right).
\end{equation}


\item If $D$ is chosen to be the ``maximal positive basis" (see~\eqref{eq:muexample}), the bound \eqref{eq:098jsusjs} reduces to $$ {\cal O}\left( \frac{n^{2} L}{\lambda} \log_2 \left( \frac{n L^2 \alpha_0^2}{\lambda \epsilon} \right) \right) = \tilde{\cal O} \left(n^2\frac{L}{\lambda}\right),$$
where the $\tilde{\cal O}$ notation suppresses the logarithmic term. The complexity is proportional to the condition number $L/\mu$.

\item As in the convex case, we can introduce the additional stopping criterion $l_k \leq \tfrac{3S}{c}$. The analysis is similar and the bound on function evaluation can be reduced by the factor of $4/3$. However, in practice we often do not know  $S$.

\end{itemize}


\section{Conclusion} \label{sec:conclusion}

We proposed and analyzed the complexity of SDS -- a simplified variant of the direct search method. Thanks to the simplified design of our method, and two novel and efficient initialization strategies, our method depends on a single parameter only. This contrasts with  standard direct search which depends on a large number of parameters. We gave the {\em first unified analysis}, covering three classes of unconstrained smooth minimization problems: non-convex, convex and strongly convex. The analysis follows the same pattern in all three cases. Finally, our complexity bounds have a simple form, are easy to interpret, and are better than existing bounds in the convex case.


\bibliography{references}

\section*{Appendix:  Proof Sketch of Proposition \ref{prop:D}}

\textbf{Proposition 7.} 
{
\em
Let $D$ be a finite set of nonzero vectors. Then the following statements are equivalent:
\begin{itemize}
\item[(i)] $D$ is a positive spanning set.
\item[(ii)] The cosine measure of $D$ is positive:
\begin{equation}\label{eq:hjs787gasassss} \mu \eqdef \min_{0\neq v\in \R^n} \max_{d\in D} \frac{\ve{v}{d}}{\|v\|\|d\|} > 0.\end{equation}
Above, $\ve{\cdot}{\cdot}$ is the standard Euclidean inner product and $\|\cdot\|$ is the standard Euclidean norm. 
\item[(iii)] The convex hull of $D$ contains $0$ in its interior.
\end{itemize}
}

\begin{proof}

(i)$\Leftrightarrow$(ii). Equivalence of (i) and (ii) can be shown by a separation argument. 

(ii)$\Rightarrow$(iii). Let ${\cal B}$ be the unit Euclidean ball in $\R^n$ centered at the origin. Note that 0 is in the interior of $Conv (D)$ if and only if it is  in the interior of ${\cal D}\eqdef Conv(D')$, where $D'=\{d/\|d\| \;:\; d\in D\}$. Further, note that in view of \eqref{eq:hjs787gasassss}, for every $0\neq v$ we have
\[\sigma_{{\cal D}}(v) \eqdef \max_{d'\in {\cal D}} \ve{v}{d'}
= \max_{d' \in D'} \ve{v}{d'} 
=\max_{d\in D} \frac{\ve{v}{d}}{\|d\|}
\overset{\eqref{eq:hjs787gasassss}}{\geq} \mu \|v\| .\]
On the other hand,  for  any $\nu>0$,  \[\sigma_{\nu {\cal B}}(v) \eqdef \max_{d'\in \nu {\cal B}} \ve{v}{d'}= \nu \|v\| .\]  If (ii) holds, then $\mu>0$, and hence $\sigma_{{\cal D}} \geq \sigma_{\mu {\cal B}}$, implying that ${\cal D}\supset \mu \cal B$, establishing (iii). 

(iii)$\Leftrightarrow$(i). If (iii) holds, then $Conv(D)\supset \nu {\cal B}$ for some $\nu>0$. Let $Cone(D)$ be the conic hull of $D$ (the smallest convex cone containing $D$). Since $Cone(D)\supset Conv(D) \supset \nu {\cal B}$,   we must have $tb\in Cone(D)$ for all $b\in {\cal B}$ and $t\geq 0$. That is, $Cone(D)=\R^n$, implying (i).
\end{proof}


\end{document}